\documentclass[10pt]{article}
\usepackage{amsmath, amsfonts, amsthm, amssymb, verbatim}
\usepackage{graphicx}
\usepackage[mathcal]{euscript}
\usepackage{amstext}
\usepackage{amssymb,mathrsfs}
\newcommand \be     {\begin{equation}}
\newcommand \ee     {\end{equation}}
\newcommand \RR      {\mathbb{R}}
\newcommand \NN      {\mathbb{N}}
\newcommand \del     \partial

\newcommand{\vp}{\varphi}

\newcommand{\clg}[1]{{\mathcal{#1}}}

\newcommand{\td}[1]{{\tilde{#1}}}

\def\XXint#1#2#3{{\setbox0=\hbox{$#1{#2#3}{\int}$}
\vcenter{\hbox{$#2#3$}}\kern-.5\wd0}}

 \newcommand{\R}{\mathbb R}

 \newcommand{\p}{\partial}

 \newcommand{\Chi}{{\bf \chi}}

\newtheorem{theorem}{Theorem}[section]
\newtheorem{proposition}[theorem]{Proposition}
 \newtheorem{remark}[theorem]{Remark}
\newtheorem{lemma}[theorem]{Lemma}
\newtheorem{corollary}[theorem]{Corollary}
\newtheorem{definition}[theorem]{Definition}

 \def\beqs{\begin{eqnarray*}}
 \def\enqs{\end{eqnarray*}}
 \def\beq{\begin{eqnarray}}
 \def\enq{\end{eqnarray}}

\begin{document}
\title{Persistence of solutions to higher order
nonlinear Schr\"odinger equation}

\author{Xavier Carvajal$^1$, Wladimir Neves$^1$}

\date{}

\maketitle

\footnotetext[1]{ Instituto de Matem\'atica, Universidade Federal
do Rio de Janeiro, C.P. 68530, Cidade Universit\'aria 21945-970,
Rio de Janeiro, Brazil. E-mail: {\sl carvajal@im.ufrj.br,
wladimir@im.ufrj.br.}
%
%
}

%

%
\begin{abstract} Applying an Abstract Interpolation Lemma, we can show
persistence of solutions of
the initial value problem
to higher order nonlinear Schr\"odinger equation, also called
Airy-Schr\"odinger equation,
in weighted
Sobolev spaces $\mathcal{X}^{2,\theta}$, for $0 \le \theta \le 1$.
\end{abstract}
%

\maketitle

\section{Introduction} \label{IN}

Motivated by the difficulty question of how to show persistent
properties of solutions to dispersive equations in the weighted
Sobolev spaces, we proved an Abstract Interpolation Lemma. Then,
applying this lemma we were able to show persistence for the so
called Airy-Schr\"odinger equation in weighted Sobolev spaces
$\mathcal{X}^{s,\theta}$, see definition in equation \eqref{SSWW},
for $0 \le \theta \le 1$ and $s=2$.

Here, we have focus on the exponent of the weighted, that is, we
have been concentrated on $\mathcal{X}^{2,\theta}$ for $\theta
<1$. In this direction our result is new, moreover to higher order
nonlinear Schr\"odinger equations. We address the reader the paper
of Nahas and Ponce \cite{NP} for similar results on the persistent
properties of solutions to semi-linear Schr\"odinger equation in
weighted Sobolev spaces. Although, that paper used different
technics from ours.


\subsection{Purpose and some results}

In this paper we describe how to obtain some new results on the
persistent properties
in weighted Sobolev spaces for solutions of the initial value problem (IVP)
\begin{equation}\label{IVP}
\begin{cases}
\p_t u + i \, a\,\p^2_xu + b\,\p^3_x u + i \,c \,|u|^{2}u + d\,
|u|^2 \p_x u+e\, u^2\p_x\bar{u}=  0, \quad (t,x) \in \R^2,\\
u(x,0)  =  u_0(x),
\end{cases}
\end{equation}
where $u$ is a complex-valued function, $a, b, c, d$ and $e$ are
real parameters and $u_0$ is a given initial-data. This model was
proposed by Hasegawa and Kodama in \cite{H-K, Ko} to describe the
nonlinear propagation of pulses in optical fibers. In literature,
it is called as a higher order nonlinear Schr\"odinger equation or
also Airy-Schr\"odinger equation. Moreover, as we are going to
show below in this first section, depending on the values of the
constants $a, b, c, d$ and $e$, \eqref{IVP} describes many
interesting known problems.

\medskip
It was shown in \cite{C2} that the flow associated to the IVP
\eqref{IVP} leaves the following quantity
\begin{equation}
I_1(v(t)) := \int_{\R}|v(x,t)|^2 \, dx, \label{CL1}
\end{equation}
conserved in time. Also, when $b e \neq 0$ we have the following
conserved quantity
\begin{equation}
\begin{aligned}
I_2(v(t)) := C_1 \, \int_{\R}|\partial_x v(x,t)|^{2}\,dx
&+ C_2 \, \int_{\R}|v(x,t)|^4 \; dx \\
&+ C_3 \, \textrm{Im} \,
\int_{\R}v(x,t)\partial_x\overline{v(t,x)}dx, \, \label{CL2}
\end{aligned}
\end{equation}
where $C_1= 3 \, b \,e$, $C_2= -e\, (e+d)/2$ and $C_3= (3\,b \, c
- a\, (e+d))$.

\bigskip
Regarding the IVP \eqref{IVP} with $b\neq 0$, Laurey in \cite{C2}
showed that the IVP is locally well-posed in $H^{s}(\R)$ with $s >
3/4$, and using the quantities \eqref{CL1} and \eqref{CL2} she
proved the global well-posedness in $H^s(\R)$ with $s\geq 1$. In
\cite{G} Staffilani established the local well-posedness in
$H^{s}(\R)$ with $s\geq 1/4$, for the IVP associated to
(\ref{IVP}), improving Laurey's result.

\medskip
In the problem \eqref{IVP}, when $a$, $b$ are real functions of
$t$, $(b\neq 0)$, was proved in \cite{CL} the local well-posedness
in $H^{s}(\R)$, for $s \geq 1/4$. Moreover, in \cite{Cv3} when $c=
(d-e) \, a / 3 \, b$, global well-posedness was established in
$H^s(\R)$ with $s> 1/4$. Also, in \cite{CM} one has the unique
continuation property for the solution of \eqref{IVP}.

\medskip
One stress the importance of the weighted Sobolev spaces. This
question goes back to work of Kato \cite{KT1}, where the space
$\mathcal{X}^{2 r, r}$ for $(r=1,2,\ldots)$ was first introduced
to prove well-posedness with weight for the KdV and generalized
KdV equations. Following Kato, we observe that functions in the
Sobolev spaces $H^s$ do not necessarily decay fast as $|x| \to
\infty$. Therefore, since we want to prove well-posedness in
spaces of fast-decaying functions, a simple choice is a weighted
Sobolev space $H^s(\RR) \cap L^2\big(\omega(x)\; dx\big)$ for some
appropriated weight function $\omega$, see \cite{KT1}.

\medskip
One of the purposes of the present paper is to show well-posedness
of \eqref{IVP}, with $b \neq 0$, in the weighted Sobolev space
$\clg{X}^{2,\theta}$, for $\theta \in [0,1]$. The most difficult
question in this work is to prove the persistence in the weighted
Sobolev space to the solution $u(t)$ of the IVP \eqref{IVP} given
by the global theory in $H^2$, which seems to be so regular, but
it is not the point of the paper. To establish this, we
approximate the solution by a sequence of smoothing solutions of
\eqref{IVP} (see Lemma \ref{otg}). We show that this sequence
belongs to a family $\clg{A}$ of functions (see conditions
\eqref{AIL0}-\eqref{AILC4}) where is possible apply the Abstract
Interpolation Lemma (Lemma \ref{1CL}), which permits to obtain the
persistence in $\clg{X}^{2,\theta}$ for this sequence of smoothing
solutions. Then passing to the limit in this sequence, we get the
persistence for the solution $u(t)$ in $\clg{X}^{2,\theta}$ as
desired.
Moreover, since $\mathcal{X}^{s,1} \subseteq
\mathcal{X}^{s,\theta}$, for all $s\in \mathbb{R}$ and $\theta\in
[0,1]$, see Remark \ref{IWSS}, we have also extended the
well-posedness results in \cite{KT1} for the KdV and generalized
KdV obtained in the weighted Sobolev space $\clg{X}^{2,1}$.
The authors would like to observe that, as it is known by them, it
is the first time in literature, where the weighted Sobolev space
$\clg{X}^{s,\theta}$ for $\theta \in [0,1]$ appears.

\medskip
An outline of this paper follows. In the rest of this section we
fix the notation, give the definition of well-posedness and
present some background concerning the theory of well-posedness
for the Airy-Schr\"odinger equation.
The Abstract Interpolation Lemma is given at Section 2. In Section
3, first we show some conserved quantities, and prove a nonlinear
estimate. Then, we formulate the approximated problems associated
to the IVP \eqref{IVP} and prove Lemma \ref{otg}, which is
important to show Theorem \ref{t1.2} at the end of this section.

\bigskip
\subsection{Notation and background}

At this point we fix some functional notation used in the paper.
By $dx$ we denote the Lebesgue measure on $\RR$ and, for $\theta
\geq 0$,
$$
  \begin{aligned}
  d\mu_\theta(x)&:= (1 + |x|^2)^ \theta \; dx,
  \\
    d\dot{\mu}_\theta(x)&:= |x|^{2 \theta} \; dx
  \end{aligned}
$$
denote the Lebesgue-Stieltjes measures on $\RR$. Hence, given a
set $X$, a measurable function $f \in L^2(X;d\mu_\theta)$ means
that
$$
  \|f\|_{L^2(X ; d\mu_\theta)}^2= \int_X |f(x)|^2 \; d\mu_\theta(x)< \infty.
$$
When $X= \RR$, we write: $L^2(d\mu_\theta) \equiv
L^2(\RR;d\mu_\theta)$, and for simplicity
$$
  L^2 \equiv L^2(d\mu_0), \quad L^2(d\mu)
  \equiv L^2(d\mu_1).
$$
Analogously, for the measure $d\dot{\mu}_\theta$. We will use the
Lebesgue space-time $L_{x}^{p}\mathcal{L}_{\tau}^{q}$ endowed with
the norm
$$
\|f\|_{L_{x}^{p}\mathcal{L}_{\tau}^{q}} = \big\|
\|f\|_{\mathcal{L}_{\tau}^{q}} \big\|_{L_{x}^{p}} = \Big(
\int_{\R} \Big( \int _{0}^{\tau} |f(x,t)|^{q} dt \Big)^{p/q} dx
\Big)^{1/p} \quad (1 \leq p,q < \infty).
$$
When the integration in the time variable is on the whole real
line, we use the notation $\|f\|_{L_{x}^{p}L_{t}^{q}}$. The
notation $\|u\|_{L^p}$ is used when there is no doubt about the
variable of integration. Similar notations when $p$ or $q$ are
$\infty$.
As usual, $H^s \equiv H^s(\RR)$, $\dot{H}^s \equiv \dot{H}^s(\RR)$
are the classic Sobolev spaces in $\RR$, endowed respectively with
the norms
$$
\|f\|_{H^s}:= \|\widehat{f}\|_{L^2(d\mu_s)},
\quad
\|f\|_{\dot{H}^s}:=\|\widehat{f}\|_{L^2(d\dot{\mu}_s)}.
$$
In this work, we study the solutions of \eqref{IVP} in the Sobolev
spaces with weight $\mathcal{X}^{s,\theta}$, defined as
\begin{equation}
\label{SSWW}
  \mathcal{X}^{s,\theta}:=H^s \cap L^2(d\mu_\theta),
\end{equation}
with the norm
$$
\| f\|_{\mathcal{X}^{s,\theta}}:=\| f\|_{H^{s}}+\|
f\|_{L^2(d\mu_\theta)}.
$$
%
%
\begin{remark}
\label{IWSS} We remark that, $\mathcal{X}^{s,1} \subseteq
\mathcal{X}^{s,\theta}$, for all $s\in \mathbb{R}$ and  $\theta\in
[0,1]$. Indeed, using H\"older's inequality
%
$$
  \|f\|_{L^2(d\dot{\mu}_\theta)} \leq \|f\|_{L^2}^{1-\theta} \;
  \|f\|_{L^2(d\dot{\mu})}^\theta.
$$
\end{remark}
%
%

\bigskip
The following definition tell us in which sense we consider the
well-posedness for the IVP \eqref{IVP}.

\begin{definition}
Let $X$ be a Banach space and $T>0$. We say that the IVP
\eqref{IVP} is locally well-posed in $X$, if the solution $u$
uniquely exists in certain time interval $[-T,T]$ (unique
existence),
describes a continuous curve in $X$ in the interval $[-T,T]$
whenever initial data belongs to $X$ (persistence), and
varies continuously depending upon the initial data (continuous
dependence), that is, continuity of the application
$$
  u_{0} \mapsto u \quad \text{from $X$ to $\mathcal{C}([-T,T];X)$}.
$$
%
Moreover, we say that the IVP \eqref{IVP} is globally well-posed
in $X$ if the same properties hold for all time $T>0$. If some
hypotheses in the definition of local well-posed fail, we say that
the IVP is ill-posed.
\end{definition}
%
\bigskip

Particular cases of \eqref{IVP} are the following:
\begin{itemize}
\item Cubic nonlinear  Schr\"odinger equation (NLS), ($a=\pm 1$,
$b=0$, $c=-1$, $d=e=0$).
\begin{align}\label{3y0}
iu_{t} \pm u_{xx} + |u|^{2}u =  0, \quad x,t \in \R.
\end{align}
The best known local result for the IVP associated to (\ref{3y0})
is in $H^{s}(\R)$, $s \geq 0$, obtained by Tsutsumi \cite{Ts}.
Since the $L^2$ norm is preserved in (\ref{3y0}), one has that
(\ref{3y0}) is globally well-posed in $H^{s}(\R)$, $s \geq 0$.
\item Nonlinear Schr\"odinger equation with derivative ($a=-1 $,
$b=0$, $c=0$, $d=2e$).
\begin{align}\label{4y0}
iu_{t} + u_{xx} + i\lambda(|u|^{2}u)_{x}=0, \quad x,t\in \R,
\end{align}
where $\lambda \in \R$. The best known local result for the IVP
associated to  (\ref{4y0}) is in $H^{s}(\R)$, $s \geq 1/2$,
obtained by Takaoka \cite{T1}. Colliander et al. \cite{C-K-S-T-T3}
proved that (\ref{4y0}) is globally well-posed in $H^{s}(\R)$, $s
>1/2$. \item Complex modified  Korteweg-de Vries (mKdV) equation
($a=0$, $b=1$, $c=0$, $d=1$, $e=0$)
\begin{align}\label{83y0}
u_{t}+u_{xxx}+|u|^{2}u_{x}=0, \quad x,t \in \R.
\end{align}
If $u$ is real, (\ref{83y0}) is the usual mKdV equation. Kenig et
al. \cite{KPV1} proved that the IVP associated to it is locally
well-posed in $H^{s}(\R)$, $s \geq 1/4$ and Colliander et al.
\cite{C-K-S-T-T2}, proved that \eqref{83y0} is globally well-posed
in $H^{s}(\R)$, $s >1/4$.

\item When $a \neq 0$ and $b=0$, we obtain a particular case of
the well-known mixed nonlinear Schr\"odinger equation
\begin{align}\label{6y0}
u_{t}=i a u_{xx}+\lambda(|u|^{2})_{x}u+g(u),\quad x,t \in \R,
\end{align}
where $g$ satisfies some appropriated conditions and $\lambda \in
\R$ is a constant. Ozawa and Tsutsumi in \cite{[O-T]} proved that
for any $\rho >0$, there is a positive constant $T(\rho)$
depending only on $\rho$ and $g$, such that the IVP (\ref{6y0}) is
locally well-posed in $H^{1/2}(\R)$, whenever the initial data
satisfies
\[\|u_{0}\|_{\mathrm{H}^{1/2}} \leq \rho.\]
\end{itemize}
There are other dispersive models similar to \eqref{IVP}, see for
instance \cite{[A],[CC],[G-T-V],[P-Na],[P-S-S-M],[SS]} and the
references therein.

\medskip
\begin{remark}
1. We can suppose $C_3 = 0$ in \eqref{CL2}. In fact, when $C_3
\neq 0$ we have the following gauge transformation
\begin{equation*}
v(x,t)=\exp\Big(ia\,x+i(a\,\alpha^2+ b\,\alpha^3) t\Big)\,
u(x+(2a\, \alpha +3b \alpha^2)t,t),
\end{equation*}
where $$\alpha= \frac{3 \, b \, c-a(d+e)}{6\, b\, e}.$$ Then,
$\;u$ solves \eqref{IVP} if and only if $\;v$ satisfies the
equation
\begin{equation*}
\p_t v+i(a+3\alpha\, b)\p_x^2 v+b\,\p_x^3
v+i(c-\alpha(e-d))\,|v|^2 v+d\,|v|^2\p_x v+ e\,v^2\p_x \bar v=0,
\end{equation*}
and in this equation we have the factor $C_3=0$.

2. Let $c= (d-e) \, a / 3 \, b$ and $u(x,t)$ be a solution of
\eqref{IVP}. If we choose a new unknown function $v(x,t)$ related
to $u$ by the relation
\begin{equation*}
v(x,t)=\exp\Big(i\frac{a}{3b}\,x+i\frac{a^3}{27b^2}\,t\Big)\,
u(x+\frac{a^2}{3b}\,t,t).
\end{equation*}
Then, $\;u$ solves \eqref{IVP} if and only if $\;v$ satisfies the
complex modified Korteweg-de Vries type equation
\begin{equation*}
\p_t v+b\,\p_x^3 v+d\,|v|^2\p_x v+ e\,v^2\p_x \bar v=0.
\end{equation*}
\end{remark}

\bigskip
\section{The Abstract Interpolation Lemma} \label{CL}

The aim of this section is to prove an interpolation lemma with
weight concerning space time-value functions.

Let $f$ be a function from $[-T,T]$ in $H^s(\RR)$ for each $T>0$
and $s>1/2$. We suppose that for all $t \in [-T,T]$, $f$ satisfies
the following conditions:

$(i)$ For each $t \in [-T,T]$, $\, \clg{L}^1\big( \xi \in \R;
f(t,\xi) \neq 0\big)>0$, where $\mathcal{L}^1\big(E\big)$ is the
Lebesgue measure of measurable set $E \subset \R$.

 $(ii)$ There exist constants $C_0, \td{C_0} >0$,
$\td{C_1} \geq 0$ (independent of $f$, $t$), such that
\begin{eqnarray}
   \label{AIL0}
   \|f(t)\|^2_{L^2} \leq C_0 \; \|f(0)\|^2_{L^2},
%
\\[5pt]
  \label{AIL1}
   \|f(t)\|^2_{L^2(d\dot{\mu})} \leq \td{C_0} \,
   \|f(0)\|^2_{L^2(d\dot{\mu})} + \td{C_1}.
\end{eqnarray}

$(iii)$ For all $\theta \in [0,1]$, there exist $\Theta>0$
(independent of $f$, $t$) and $\gamma_1 \in (0,1)$, such that
\begin{equation}
\label{AILC3}
    \int_{\{|f(t)|^2 < \Theta\}} |f(t)|^2 \; d\dot{\mu}_\theta
    \leq \gamma_1 \int_\RR |f(t)|^2 \;
    d\dot{\mu}_\theta.
\end{equation}

$(iv)$ There exist $R>0$ and $\gamma_2 \in (0,1)$ (both
independent of $f$), such that
\begin{equation}
\label{AILC4}
    \int_{\{\RR \setminus (-R,R)\}} |f(0)|^2 \; d\dot{\mu}
    \leq \gamma_2 \int_\RR |f(0)|^2 \;
    d\dot{\mu}.
\end{equation}

We denote by $\clg{A}\,$ a set of functions that satisfies the
above conditions. The following remark shows a non-enumerable
number of non-empty sets $\clg{A}$.

\begin{remark}
Let $R_0,T>0$ be constants and $b>0$, such that for each $\theta
\in [0,1]$,\be\label{eq12} \int_{\{R_0 \le |\xi|\le R_0+b\}}
\xi^{2 \theta}d\xi \le \dfrac{1}{2(T+1)^2}\int_{\{0 \le |\xi|\le
R_0\}} \xi^{2 \theta}d\xi. \ee Let $\mathcal{A}_0^A$ the set of
the continuous functions in $\R$ such that
 $f(\xi)=0$ if $|\xi|> R_0+b$,
 $f(\xi)=A$ if $|\xi| \le R_0$ and $0\le f(\xi) \le A$, where $A$ is any positive real number, fixed.
 Now, we set
 $$
 \mathcal{A}_1^A:=\left\{f(t,\xi)=f(\xi)(1+|t|); \; t \in [-T,T], f(\xi) \in \mathcal{A}_0^A\right\}.
 $$
Then, for each $A$, $\mathcal{A}_1^A$ is a like set $\mathcal{A}$.
In fact, condition $(i)$ is clearly satisfied. The condition (ii)
is satisfied with $C_0=\tilde{C_0}=1+T$. The condition (iv) is
satisfied with $R=R_0+b$ for all $\gamma_2 \in (0,1)$, since the
first integral in \eqref{AILC4} is null. And the condition (iii)
is satisfied with $\Theta=A^2$ and $\gamma_1=1/2$, since
\eqref{eq12} implies
 \begin{align*}
 \int_{\{ |f(t)|^2< A^2\}} \xi^{2 \theta}|f(t,\xi)|^2 d\xi \le & \,(1+T)^2 A^2  \int_{\{ R_0 \le |\xi|\le R_0+b \}} \xi^{2 \theta}d\xi\\
 \le &\, \dfrac{(1+T)^2 A^2  }{2(1+T)^2}\int_{\{ 0 \le |\xi|\le R_0 \}} \xi^{2 \theta} d\xi\\
 \le & \, \dfrac12 \int_{\{ 0 \le |\xi|\le R_0 \}} \xi^{2 \theta} |f(t,\xi)|^2d\xi\\
 \le & \, \dfrac12 \int_{ \R } \xi^{2 \theta} |f(t,\xi)|^2d\xi.
 \end{align*}
\end{remark}

\begin{lemma}\label{1CL}
For each $\theta \in (0,1)$, there exists a positive constant
$\rho(\theta)$, such that, for each $t \in [-T,T]$
\be \label{AIL}
    \|f(t)\|^2_{L^2(d\dot{\mu}_\theta)} \leq \|f(t)\|^{2 \rho}_{H^s} \;
    \Big( K_0 \, \|f(0)\|^2_{L^2}
    + K_1 \, \|f(0)\|^2_{L^2(d\dot{\mu}_\theta)}
    + K_2 \Big)
\ee
for all $f \in \clg{A}$, where
$$
  K_0= C_0 \,  R^{2 \theta} \, \left(\frac{4}{\Theta}\right)^{\rho +
  1},\quad
  K_1= \frac{\tilde{C}_0}{\rho (1-\gamma_2)} \,
  \left(\frac{4}{\Theta}\right)^{\rho}, \quad
  K_2= \frac{\tilde{C}_1}{\rho R^{2 \theta \rho}}.
$$
\end{lemma}

\begin{proof}
1. For simplicity of notation, we sometimes write $f(t,\xi) \equiv
f(\xi)$ and $f(0,\xi) \equiv f_0(\xi)$. Let $\theta_j>0$,
$(j=0,1)$, constants independents of $t$, and for $\theta \in
[0,1]$ set
\begin{eqnarray*}
I_1^{\theta_1}:= \int_\RR |\xi|^{2 \theta} \, |f(\xi)|^2 \,
\chi_{\{|f(\xi)|^2>\theta_1\}} \; d\xi
\\
I_2^{\theta_1}:=\, \theta_1 \int_\RR |\xi|^{2 \theta} \,
\chi_{\{|f(\xi)|^2>\theta_1\}} \; d\xi,
\\
I_3^{\theta_1}:= \int_\RR |\xi|^{2 \theta} \, |f(\xi)|^2 \,
\chi_{\{|f(\xi)|^2 \leq \theta_1\}} \; d\xi,
\end{eqnarray*}
where $\chi_A$ is the characteristic function of the set $A$.
Therefore, we have
$$
  I:= \int_\RR |\xi|^{2 \theta} \,
|f(\xi)|^2  \; d\xi = I_1^{\theta_1} + I_3^{\theta_1} =
I_1^{\theta_1} - \theta_0 I_2^{\theta_1} + I_3^{\theta_1} +
\theta_0 I_2^{\theta_1}.
$$
Moreover, it is clear that $I_2^{\theta_1} < I_1^{\theta_1}$,
indeed
$$
  I_1^{\theta_1} - I_2^{\theta_1} =  \int_\RR |\xi|^{2 \theta} \, \big(
  |f(\xi)|^2 - \theta_1 \big) \, \chi_{\{|f(\xi)|^2 >
  \theta_1\}}
\; d\xi >0.
$$
Hence, $\theta_0 I_2^{\theta_1} < \theta_0 I_1^{\theta_1} <
\theta_0 \, (I_1^{\theta_1}+I_3^{\theta_1}) = \theta_0 \, I$.
Therefore, we have
\begin{equation}
\label{I11}
      (1 - \theta_0) \, I < I - \theta_0 I_2^{\theta_1} =
      I_1^{\theta_1} - \theta_0 I_2^{\theta_1}+ I_3^{\theta_1}.
\end{equation}

2. Claim $\sharp 1$: There exist $\theta_1>0$ independent of $f$,
$t \in [-T,T]$, and a positive constant $\beta <1$, such that
$I_3^{\theta_1} < \beta I_1^{\theta_1}$.

Proof of Claim $\sharp 1$: We must show that
$$
  \begin{aligned}
  \int_\RR |\xi|^{2 \theta} \, |f(t,\xi)|^2 \,
  \chi_{\{|f|^2 \leq \theta_1\}} \; d\xi
  &\leq \beta \,
  \int_\RR |\xi|^{2 \theta} \,
|f(t,\xi)|^2 \, \chi_{\{|f|^2 > \theta_1\}} \; d\xi
\\
  &= \beta \,
  \int_\RR |\xi|^{2 \theta} \,
|f(t,\xi)|^2 \; d\xi
\\
  &- \beta \,
  \int_\RR |\xi|^{2 \theta} \,
|f(t,\xi)|^2 \, \chi_{\{|f|^2 \leq \theta_1\}} \; d\xi.
\end{aligned}
$$
Therefore, it is enough to show that
$$
  \begin{aligned}
    \int_\RR |\xi|^{2 \theta} \, |f(t,\xi)|^2 \,
  \chi_{\{|f|^2 \leq \theta_1\}} \; d\xi
  \leq \frac{\beta}{1 + \beta} \; \int_\RR |\xi|^{2 \theta} \,
|f(t,\xi)|^2 \; d\xi,
\end{aligned}
$$
which is satisfied since $f \in \clg{A}$. Consequently, we take
$\theta_1= \Theta$ of inequality \eqref{AILC3}.


\bigskip
3. Now from item 2, we are going to show the existence of a
positive constant $\alpha < 1/2$, such that
\be \label{I12}
    I_3^{\theta_1} < \alpha (I_1^{\theta_1} + I_3^{\theta_1}) = \alpha I.
\ee
Indeed, we have
$$
\begin{aligned}
      I_3^{\theta_1} < \alpha (I_1^{\theta_1} + I_3^{\theta_1})
      & \Leftrightarrow (1-\alpha) I_3^{\theta_1} < \alpha \, I_1^{\theta_1}   \\
      & \Leftrightarrow
      I_3^{\theta_1} < \frac{\alpha}{1-\alpha} \, I_1^{\theta_1}.
\end{aligned}
$$
Therefore, it is enough to take $\beta < 1$ and, we have
$$
  0 < \alpha = \frac{\beta}{1+\beta} < \frac{1}{2}.
$$
Now, we fix $\theta_0 = (3/4 - \alpha)>1/4$ and, from \eqref{I11},
\eqref{I12}, we obtain
\be \label{I13} I < \frac{I_1^{\theta_1} - \theta_0
I_2^{\theta_1}}{1 - (\theta_0 + \alpha)}=4\left(I_1^{\theta_1} -
\theta_0 I_2^{\theta_1}\right). \ee

\bigskip
4. Claim $\sharp 2$: There exist $N_1 \in \NN$ and a constant
$C_1>0$ both independent of $f$ and $t$, such that, for all $\eta
\geq N_1$
$$
  \int_{\{ |\xi| < \eta\}} |f(\xi)|^2 |\xi|^2 \; d\xi \leq C_1
  \int_{\{ |\xi| < \eta\}} |f_0(\xi)|^2 |\xi|^2 \; d\xi+\td{C_1}.
$$

Proof of Claim $\sharp 2$: Equivalently, we have to show that
$$
  \begin{aligned}
  \int_\RR |f(\xi)|^2 |\xi|^2 \; d\xi &- \int_{\{ |\xi| \geq \eta\}} |f(\xi)|^2 |\xi|^2 \; d\xi
  \\
  &\leq C_1 \int_\RR |f_0(\xi)|^2 |\xi|^2 \; d\xi -
  C_1 \int_{\{|\xi| \geq \eta\}} |f_0(\xi)|^2 |\xi|^2 \;
  d\xi+\td{C_1},
  \end{aligned}
$$
for each $\eta \geq N_1$. Hence using \eqref{AIL1} and supposing
$C_1
> \td{C_0}$, it is sufficient to prove that
$$
  \begin{aligned}
  \td{C_1} + \td{C_0} \int_\RR |f_0(\xi)|^2 |\xi|^2 \; d\xi &-
  \int_{\{|\xi| \geq \eta\}} |f(\xi)|^2 |\xi|^2 \; d\xi
  \\
  &\leq C_1 \int_\RR |f_0(\xi)|^2 |\xi|^2 \; d\xi -
  C_1 \int_{\{|\xi| \geq \eta\}} |f_0(\xi)|^2 |\xi|^2 \;
  d\xi + \td{C_1}.
  \end{aligned}
$$
By a simple algebraic manipulation, it is sufficient to show that
$$
  \begin{aligned}
   C_1 \int_{\{ |\xi| \geq \eta\}} |f_0(\xi)|^2 |\xi|^2 \;
  d\xi &\leq \big( C_1-\td{C_0} \big) \int_\RR |f_0(\xi)|^2 |\xi|^2 \;
  d\xi
  \\
  &+ \int_{\{|\xi| \geq \eta\}} |f(\xi)|^2 |\xi|^2 \;
  d\xi.
  \end{aligned}
$$
Therefore, it is enough to show that
$$
  \int_{\{|\xi| \geq \eta\}} |f_0(\xi)|^2 |\xi|^2 \;
  d\xi \leq \frac{C_1-\td{C_0}}{C_1} \int_\RR |f_0(\xi)|^2 |\xi|^2 \;
  d\xi,
$$
which it is true for $f \in \clg{A}$. Consequently, we take $N_1=
R$ of inequality \eqref{AILC4}.

\bigskip
5. Finally, we estimate $I_1^{\theta_1} - \theta_0
I_2^{\theta_1}$.
\\
If $\theta=0,1$ by \eqref{AIL0} and \eqref{AIL1} is obvious that
$$
I_1^{\theta_1} - \theta_0 I_2^{\theta_1} \le C_0 \int_\RR |\xi|^{2
\theta} \,
      |f_0(\xi)|^{2 \theta} \; d\xi.
$$
Consequently, we consider in the following $\theta \in (0,1)$.
$$
\begin{aligned}
      I_1^{\theta_1} - \theta_0 I_2^{\theta_1}& = \int_\RR \big(|\xi|^{2 \theta} \,
      |f(\xi)|^2 - \theta_0 \theta_1 |\xi|^{2 \theta} \big) \,
\chi_{\{|f(\xi)|^2>\theta_1\}} \; d\xi
\\
      &= \int_\RR \big(\,(\,|\xi| \,
      |f(\xi)|^{1/\theta} \,)^{2 \theta} - ((\theta_0 \theta_1)^{1/2 \theta} |\xi| \,)^{2 \theta} \big) \,
\chi_{\{|f(\xi)|^2>\theta_1\}} \; d\xi.
\end{aligned}
$$
For each $\eta>0$, let $\varphi(\eta) = \eta^{2 \theta}$. Hence,
$\varphi'(\eta)= 2 \, \theta \; \eta^{2 \theta-1}> 0$ and, we have
$$
\begin{aligned}
      I_1^{\theta_1} - \theta_0 I_2^{\theta_1}& = \int_\RR \int_{(\theta_0 \theta_1)^{1/2 \theta}
      |\xi|}^{|\xi| \, |f(\xi)|^{1/\theta}}
      \varphi'(\eta) \,  d\eta \; d\xi
\\
      &= 2 \theta \int_0^\infty \eta^{2 \theta-1} \int_\RR
\chi_{E(\eta)}(\xi) \; d\xi \, d\eta
\\
      &= 2 \theta \int_0^\infty \eta^{2 \theta-1} \; \mathcal{L}^1\big((E(\eta)\big) \,
      d\eta,
\end{aligned}
$$
where
$$
  E(\eta)= \big\{ \xi \in \RR \,/\,
  |f(\xi)|^{1/\theta} |\xi| > \eta \, \big\} \bigcap
  \big\{ \xi \in \RR \,/\, \, \kappa \, |\xi| < \eta, \; \kappa = (\theta_0 \theta_1)^{1/2 \theta}
  \big\}.
$$
We observe that,
$$
\begin{aligned}
  \mathcal{L}^1\big((E(\eta)\big) &\leq \int_{\{\kappa \, |\xi| < \eta\}}
  \frac{|f(\xi)|^{2/\theta} |\xi|^{2}}{\eta^{2}}
  \; d\xi.
\end{aligned}
$$
%
%
%
Hence we obtain
$$
\begin{aligned}
      I_1^{\theta_1} - \theta_0 I_2^{\theta_1}& \leq  2 \theta
      \int_0^\infty \eta^{2 \theta-1}
      \int_{\{\kappa \, |\xi| < \eta\}} \frac{|f(\xi)|^{2/\theta} \, |\xi|^{2}}{\eta^2} \; d\xi \, d\eta
\\
      &\leq 2 \theta \; \|f\|_{H^s}^{(2/\theta) - 2} \int_0^\infty \eta^{2 \theta - 3}
      \int_{\{\kappa \, |\xi| < \eta\}} |f(\xi)|^{2} \, |\xi|^{2}  \; d\xi \,
      d\eta.
\end{aligned}
$$
From item 4 and applying \eqref{AIL0}, it follows that
$$
\begin{aligned}
      I_1^{\theta_1} - \theta_0 I_2^{\theta_1}& \le 2 \theta \; \|f(t)\|_{H^s}^{(2/\theta) - 2}
      \int_0^{N_1} \eta^{2 \theta-3}
      \int_{\{\kappa \, |\xi| < \eta \}} |f(\xi)|^{2} \, \frac{\eta^{2}}{\kappa^2} \;  \; d\xi \,
      d\eta
\\
     &+ 2 \theta \; \|f(t)\|_{H^s}^{(2/\theta) - 2}
      \int_{N_1}^\infty \eta^{2 \theta - 3}
      \int_{\{\kappa \, |\xi| < \eta\}} |f(\xi)|^{2} \, |\xi|^2 \;  \; d\xi \,
      d\eta
\\
      &\leq \frac{C_0 \;N_1^{2 \theta}}{\kappa^2} \; \|f(t)\|_{H^s}^{(2/\theta) - 2} \,
      \int_{\R} |f_0(\xi)|^{2} \; d\xi
\\
      &+ 2 \theta \; C_1 \,\|f(t)\|_{H^s}^{(2/\theta) - 2}  \,
      \int_\RR |f_0(\xi)|^{2} \, |\xi|^2
      \int_{\{ \eta > \kappa \, |\xi| \}} \eta^{2 \theta - 3}  \;  \; d\eta \,
      d\xi + \Xi
\\
      &= \frac{C_0 \;N_1^{2 \theta}}{\kappa^2} \; \|f(t)\|_{H^s}^{(2/\theta) - 2} \,
      \int_{\R} |f_0(\xi)|^{2} \;  \; d\xi
\\
      &+ \frac{\theta}{1-\theta} \; C_1 \,\|f(t)\|_{H^s}^{(2/\theta) - 2}  \,
      \int_\RR |f_0(\xi)|^{2} \, |\xi|^2
      \; |\xi|^{2 \theta - 2} \kappa^{2 \theta - 2} \,
      d\xi + \Xi
\\[5pt]
      &= \left(\frac{\;4 }{\theta_1}\right)^{1/\theta} C_0 \; N_1^{2\theta} \;
      \|f(t)\|_{H^s}^{(2/\theta) - 2} \,
      \int_{\R} |f_0(\xi)|^{2} \; d\xi+
\\
      &+ \left(\frac{4}{\theta_1}\right)^{(1-\theta)/\theta}\, \frac{C_1 \; \theta}{1 - \theta}
      \, \|f(t)\|_{H^s}^{(2/\theta) - 2} \,
      \int_\RR |f_0(\xi)|^{2}
      \; |\xi|^{2\theta} \,
      d\xi + \Xi,
\end{aligned}
$$
where
$$
  \Xi= \dfrac{\theta \td{C_1}\|f(t)\|_{H^s}^{(2/\theta) - 2}}{(1-\theta)N_1^{2(1-
      \theta)}}.
$$
\end{proof}
%
%

\section{ Statement of the well-posedness result}

This is the section where the well-posedness of the Cauchy problem
\eqref{IVP} in weighted Sobolev space $\clg{X}^{2,\theta}$, for
$\theta \in [0,1]$ is proved.

\subsection{A priori estimates} \label{CL0}

\begin{lemma} \label{LPE} If $u(t)$ is a solution of the IVP \eqref{IVP} with $u(0)$ in $H^2$,
then for each $T>0$,
\beq \label{CQ1}
  \|u(t)\|_{L^2}= \|u(0)\|_{L^2},
\enq
\begin{align} \label{CQ2}
  \|u_x(t)\|_{L^2} \le  2\|u_x(0)\|_{L^2} +  \,C\, |C_2/C_1| \,
  \|u(0)\|_{L^2}^3
  +2|C_3/C_1|\, \|u(0)\|_{L^2},
\end{align}
for all $t \in [-T,T]$. Moreover, we have
\beq \label{CQ3}
   \|u_{xx}(t)\|_{L^2}^2 \le
   \big(\|u_{xx}(0)\|^2_{L^2} + \hbar\big) \, (1+T) \, e^{\hbar T},
\enq
where  $\hbar = \hbar(\|u_0\|_{L^2}, \|u_{0x}\|_{L^2})$ is a
continuous function with $\hbar(0,0)=0$.
\end{lemma}
\begin{proof}
The inequalities \eqref{CQ1} and \eqref{CQ2} are consequence of
conserved laws \eqref{CL1}, \eqref{CL2} and the Gagliardo-
Nirenberg inequality $\|v\|_{L^4} \le
\|v\|_{L^2}^{3/4}\|v_x\|_{L^2}^{1/4}$, see \cite{Cv4} and
\cite{C2}. For inequality \eqref{CQ3}, we address \cite{C2}.
\end{proof}

\begin{proposition}\label{xav1}
Let $u$ be a solution of the IVP \eqref{IVP}. If $u(t) \in
\mathcal{X}^{s,1}$ for each $t \in [-T,T]$, with $s \ge 3$, then
%
\begin{equation*} 
    \|u(t)\|^2_{L^2(d\dot{\mu})} \le  \, e^{ a_0 T}\,
    (\|u_0\|^2_{L^2(d\dot{\mu})}+a_0 \hbar_0 T),
\end{equation*}
for all $t \in [-T,T]$, where $a_0:= 2 |a|+ 3 |b|+ (|d+e|)/2$ and
$$
  \hbar_0= \hbar_0(\|u(0)\|_{L^2},\|u_x(0)\|_{L^2},\|u_{xx}(0)\|_{L^2})
$$
is a continuous function with $\hbar_0(0,0,0)=0$.
\end{proposition}
\begin{proof}
1. First, let us consider a convenient function, i.e. $\vp_n \in
C^\infty(\R)$, $\vp_n$ a non-negative even function, such that,
for each $x \geq 0$, $0\le \vp_n(x) \le x^2$, $0 \le \vp_n'(x) \le
2x$ and $|\vp_n^{(j)}(x)|\le 2$, $(j=2,3)$. Moreover, for $0 \le x
\le n$, $\vp_n(x)=x^2$, and for $x > 10 \, n$, $\vp_n(x) = 10 \,
n^2$.

\medskip
2. Now, multiplying the equation \eqref{IVP} by $\vp_n \bar{u}$
and taking the real part, we get after integration by parts,
\be \label{PX}
\begin{aligned}
\dfrac{\partial}{\partial t} \int_{\R} &\vp_n|u|^2 dx +2 a \, \Re
\left(i
 \int_{\R} \vp_n\bar{u} u_{xx}dx\right)
 \\
&+2b \, \Re\left( \int_{\R}\vp_n\bar{u}
 u_{xxx}dx\right)+ 2c \, \Re \left(i \int_{\R}\vp_n
 |u|^4dx\right) \\
&+  2d \, \Re\left( \int_{\R} \vp_n\bar{u}
 |u|^2u_{x}dx\right)+2e \, \Re\left( \int_{\R}\vp_n\bar{u}
 u^2\bar{u}_{x}dx\right)=0.
\end{aligned}
\ee
The term with coefficient $c$ is zero. Integrating by parts two
times the integral with coefficient $b$, we obtain
\begin{align}\label{px0}
\int_{\R}\vp_n\bar{u} u_{xxx}dx=\int_{\R}\vp_n\bar{u}_{xx}
u_{x}dx+2\int_{\R}\vp_n'\bar{u}_x u_{x}dx+\int_{\R}\vp_n''
u_{x}\bar{u}dx.
\end{align}
Integrating by parts the first term in the right-hand side of
\eqref{px0}
\beq\label{px2} \int_{\R}\vp_n\bar{u}_{xx} u_{x}dx
=-\int_{\R}\vp_nu \bar{u}_{xxx}dx-\int_{\R}\vp_n' u
\bar{u}_{xx}dx, \enq
and integrating by parts the second term in the right-hand side of
\eqref{px0}, we have
\beq\label{px3} 2\int_{\R}\vp_n' \bar{u}_{x} u_x dx=- 2\int_{\R}
\vp_n'' \bar{u}_{x}u dx- 2\int_{\R} \vp_n' \bar{u}_{xx}u dx. \enq
Now, combining the equations \eqref{px0}-\eqref{px3}, we get
\begin{align*}
\int_{\R}\vp_n\bar{u} u_{xxx} dx=-\int_{\R}\vp_nu \bar{u}_{xxx} dx
&-3\int_{\R}\vp_n'\bar{u}_{xx} u dx \\
&-2\int_{\R}\vp_n''\bar{u}_x u dx+\int_{\R}\vp_n''u_{x}\bar{u} dx,
\end{align*}
and thus
\begin{align}\label{PX4}
2\Re\left(\int_{\R}\vp_n\bar{u} u_{xxx}dx\right)= -3\Re
\int_{\R}\vp_n'\bar{u}_{xx} u dx-\Re\int_{\R}\vp_n''\bar{u}_x u
dx.
\end{align}
Integrating by parts the integral with coefficient $a$ in
\eqref{PX}
\beqs
\int_{\R}\vp_n  u_{xx}\bar{u}dx=- \int_{\R} \vp_n
\bar{u}_{x}u_{x}dx- \int_{\R} \vp_n' \bar{u}u_{x} dx. \enqs
Therefore, \beq \label{px4}
   \Re \left(i  \int_{\R} \vp_n\bar{u}
   u_{xx}dx\right)= \Im \left(\int_{\R} \vp_n' \bar{u}u_{x}dx\right).
\enq
Now, we consider the integral with coefficient $d$ in \eqref{PX}
and integrating by parts, we have
\beqs \int_{\R} \vp_n\bar{u} |u|^2u_{x}dx=- \int_{\R} \vp_n
u\bar{u}^2 u_{x}dx-2\int_{\R} \vp_nu^2\bar{u}
\bar{u}_{x}dx-\int_{\R} \vp_n'  u^2\bar{u}^2dx, \enqs
and this inequality implies \beq \label{px5}
  \Re\left(\int_{\R}
  \vp_n\bar{u} |u|^2u_{x}dx\right)= - \frac14 \int_{\R} \vp_n'  |u|^4dx.
\enq
Finally, we consider the last term, that is, with coefficient $e$
\beq \label{px6}
  \Re\left(\int_{\R} \vp_n u^2\bar{u}
  \bar{u}_{x}dx\right)= \Re\left(\int_{\R} \vp_n \bar{u} |u|^2
  {u}_{x}dx\right)=-\frac14 \int_{\R} \vp_n' |u|^4dx.
\enq

3. From \eqref{PX} and \eqref{PX4}-\eqref{px6}, we obtain
\begin{align*}
\dfrac{\partial}{\partial t} \int_{\R} \vp_n|u|^2dx=& -2 a \, \Im
\left(\int_{\R} \vp_n' \bar{u}u_{x}dx\right)
+3b \, \Re \int_{\R}\vp_n'\bar{u}_{xx} u dx\nonumber \\
& + b \, \Re\int_{\R}\vp_n''\bar{u}_x u dx+
\dfrac{(d+e)}{2}\int_{\R} \vp_n'  |u|^4 dx,
\end{align*}
and using the properties of $\vp_n$, we get
\begin{align*}
\dfrac{\partial}{\partial t} \int_{\R} \vp_n|u|^2dx\le & 2 |a| \int_{\R} x^2 |u|^2 dx +2 |a|\int_{\R}|u_{x}|^2 dx+3|b| \int_{\R} x^2 |u|^2 dx\nonumber\\
& +3|b|\int_{\R}|u_{xx}|^2 dx+ |b|\int_{\R}|u|^2 dx + |b|\int_{\R}|u_{x}|^2 dx\nonumber\\
&+\dfrac{|d+e|}{2}\int_{\R} x^2  |u|^2
dx+\dfrac{|d+e|}{2}\int_{\R} |u|^6 dx.
\end{align*}
Now, passing to the limit as $n \to \infty$ and applying the
Dominated Convergence Theorem
$$
\begin{aligned}
    \dfrac{\partial}{\partial t} \|u(t)\|^2_{L^2(d\dot{\mu})}
    &\leq a_0 \; \big( \|u(t)\|^2_{L^2(d\dot{\mu})} + A \big),
%
\end{aligned}
$$
where $A= \|u(0)\|^2_{L^2} + (1+\|u(0)\|^4_{L^2}) \sup_{t \in \RR}
\|u_x(t)\|^2_{L^2} + \sup_{t \in \RR} \|u_{xx}(t)\|^2_{L^2}$.
Observe that by \eqref{CQ2} and \eqref{CQ3}
$$
A \le
\hbar_0(\|u(0)\|_{L^2},\|u_x(0)\|_{L^2},\|u_{xx}(0)\|_{L^2})<
\infty,
$$
where $\hbar_0$ is a continuous function with $\hbar_0(0,0,0)=0$.
Now, applying Gronwall's inequality, we have for all $t \in [0,T]$
\begin{align}
\label{px8}
  \|u(t)\|^2_{L^2(d\dot{\mu})} \leq e^{a_0 T} \Big( \|u(0)\|^2_{L^2(d\dot{\mu})}
  + a_0 \; T \, A \Big).
%
%
%
\end{align}

\medskip
4. Let $k_0$ be a non-zero real parameter and set
$\tilde{u}(x,t):= u(x,k_0 \, t)$. Since $u$ is a solution of
\eqref{IVP} for every $t \in \R$, then $\tilde{u}$ is a global
solution of the following Airy-Schr\"odinger equation
\begin{equation*}
\partial_t \tilde{u}+ \, i \, \tilde{a} \, \partial^2_x \tilde{u} + \tilde{b} \,
\partial^3_x \tilde{u}
+i \, \tilde{c} \, |\tilde{u}|^{2} \tilde{u} + \tilde{d} \,
|\tilde{u}|^2
\partial_x \tilde{u}
+ \tilde{e} \, \tilde{u}^2 \partial_x \bar{\tilde{u}}=  0,
\end{equation*}
where $\tilde{a} = k_0 \, a, \ldots, \tilde{e} = k_0 \, e$.
Therefore, we have an analogously inequality for $\tilde{u}$, that
is
$$
    \|\tilde{u}(t)\|^2_{L^2(d\dot{\mu})} \leq e^{\tilde{a}_0 \tilde{T}} \Big( \|u(0)\|^2_{L^2(d\dot{\mu})}
  + \tilde{a}_0 \; \tilde{T} \, A \Big),
$$
for all $t \in [0,\tilde{T}]$, where $\tilde{a_0}= |k_0| \, a_0$.
Now, taking $\tilde{T}= T$ and $k_0=-1$, we obtain that the
inequality \eqref{px8} is valid for all $t \in [-T,T]$.

\end{proof}

\bigskip
\subsection{Unitary group and non-linear estimate}

We begin defining the unitary group $U(t)$ as the solution of the
linear initial value problem associated to \eqref{IVP},
\begin{equation*}
\begin{cases}
\partial_t u + i a \,\partial_x^2 \, u+   b \,\partial_x^3 u=0,\\
u(x,0)  = \,u_0(x).
\end{cases}
\end{equation*}
Hence, we have
\begin{align}\label{grupoU}
\widehat{U(t)u_0}(\xi)=\exp\big( {it(a\xi^2 +b\xi^3) }\big) \,
\widehat{u_0}(\xi).
\end{align}
For convenience, we define the non-linear part of equation in
\eqref{IVP} as
\begin{equation}
\label{NLF}
  F(u):= i c \,|u|^{2}u+d\, |u|^2 \partial_x u+e\,
u^2\partial_x\bar{u}.
\end{equation}

Next we recall a well known result, see for instance \cite{CL,
KPV1}.

If $f \in L_x^1 \mathcal{L}_t^2$, $u_0 \in \dot{H}^{1/4}$ and
$U(t)$ is the unitary group as in \eqref{grupoU}. Then, there
exists a constant $C>0$, such that
\begin{align}\label{4tt}
\left\| \partial_x \int_0^t U(t-t')f(x,t')dt'\right\|_{L_x^2} \le
C \, \|f\|_{L_x^1 \mathcal{L}_t^2}
\end{align}
and
\begin{equation}\label{efm}
\| U(t')u_0\|_{L_x^4 L_t^\infty} \leq
C \,\|u_0\|_{\dot{H}^{1/4}}.
\end{equation}
%
%
This result enable us to prove the following
\begin{proposition}\label{eq:x33k.2}
Let $u\in \mathcal{C}(\R, H^2)$ be the solution of IVP
\eqref{IVP}, then
\begin{align}\label{esmax22}
\|u\|_{L_x^4\mathcal{L}_t^{\infty}} \le & \, C \;
\|u(0)\|_{\dot{H}^{1/4}}+ C \; \int_{0}^{t}(\|u(t')\,\|_{H^{1/2+}}\|u(t')\,\|_{H^{2}}^2\nonumber\\
& +\|u(t')\,\|_{H^{1/2+}}^2\|u(t')\,\|_{H^{2}})dt',
\end{align}
where $C$ is a positive constant.
\end{proposition}

\begin{proof}
In order to prove this inequality we rely on the integral equation
form
\begin{align*}
u(t)=& \,U(t)u_0 - \int_0^{t}U(t-\tau)F(u)(\tau)d\tau,
\end{align*}
where $F(u)$ is given by \eqref{NLF}. The linear estimate
\eqref{efm} shows that if $u(0) \in H^2$ then for any $t>0$
\begin{align}\label{esmax2}
\|u\|_{L_x^4\mathcal{L}_t^{\infty}} \le C\|u(0)\|_{\dot{H}^{1/4}}+
C \int_{0}^{t}(\|F(u)\|_{L_x^2}+\|\,\partial_xF(u)\|_{L_x^2})dt'.
\end{align}
First, we estimate $\|F(u)\|_{L_x^2}$. By the immersions
$\|u(t)\|_{L_x^{\infty}} \le C\| u(t)\|_{H^{1/2+}}$ and
$\|u(t)\|_{L_x^{4}} \le  C\| u(t)\|_{\dot{H}^{1/4}}$, it follows
that
\begin{align}\label{esmax3}
\|\,|u|^2u(t')\,\|_{L_x^2}\le & \|u(t')\,\|_{L_x^{\infty}}\|u^2(t')\,\|_{L_x^2}\le C\|u(t')\,\|_{H^{1/2+}}\|u(t')\,\|_{L_x^{4}}^2\nonumber\\
\le & C\|u(t')\,\|_{H^{1/2+}}\|u(t')\,\|_{\dot{H}^{1/4}}^2,
\end{align}
and
\begin{align}\label{esmax23}
\|\,|u|^2u_x(t')\,\|_{L_x^2}\le \|u(t')\,\|_{L_x^{\infty}}^2
\|u_x(t')\,\|_{L_x^2}\le
C\|u(t')\,\|_{H^{1/2+}}^2\|u(t')\,\|_{\dot{H}^{1}}.
\end{align}
Analogously, we treat the term $u^2\partial_x\bar{u}$.

Now consider $\|\partial_xF(u)\|_{L_x^2}$, we estimate the term
$|u|^2u_x$. The estimates for the other terms are similar. Using
Leibniz rule, it is easy to see that
\begin{align}\label{esmax4}
\|\partial_x(|u|^2u_x)(t')\|_{L_x^2}\le &\, \|\bar{u} u_x^2(t')\,\|_{L_x^{2}}+\|\,|u_x|^2u(t')\,\|_{L_x^2}+\|\,|u|^2u_{xx}(t')\,\|_{L_x^2}\nonumber\\
\le &
\,C\|u(t')\,\|_{H^{1/2+}}\|u(t')\,\|_{H^{2}}^2+\|u(t')\,\|_{H^{1/2+}}^2\|u(t')\,\|_{H^{2}}.
\end{align}
Hence combining \eqref{esmax2}-\eqref{esmax4} we conclude
\eqref{esmax22}.
\end{proof}

\bigskip
\subsection{Approximated problem}

Let $X$ be a Banach space, $u_0 \in X$ and
$(u_0^{\lambda})_{\lambda>0}$ a family of regular functions, such
that
$$
  u_0^{\lambda} \to u_0 \quad \text{in $X$},
$$
when $\lambda \to \infty$. For each $\lambda>0$, we consider the
following family of approximated problems obtained from
\eqref{IVP}
\begin{equation}\label{reivaxIVP}
\begin{cases}
\partial_t u^{\lambda}+ia\,\partial^2_xu^{\lambda}+b\,\partial^3_x u^{\lambda}+ic\,|u^{\lambda}|^{2}u^{\lambda}+d\,
|u^{\lambda}|^2 \partial_x u+e\, {u^{\lambda}}^2\partial_x\bar{u}^{\lambda}=  0, \quad x,t \in \R,\\
u^{\lambda}(x,0)  =  u_0^{\lambda}(x)
.
\end{cases}
\end{equation}

As mentioned in the Introduction, we know that \eqref{IVP} is
global well-posedness in $H^2$.
In order to prove the well-posedness with weight result, Theorem
\ref{t1.2}, we initially proving the following

\begin{lemma}\label{otg}Let $T>0$, $u_0 \in H^{2}$, $u_0^{\lambda} \to u_0$
in $H^2$, and any $s \in [0,2)$ fixed. Then, for each $t \in
[-T,T]$, the family $(u^{\lambda})(t)$, solutions of the
approximated problems \eqref{reivaxIVP}, converges to $u(t)$ in
$H^{s}$, uniformly with respect to $t$, where $u(t) \in H^2$ is
the global solution of the IVP \eqref{IVP}.
\end{lemma}
\begin{proof}
1. We begin proving that $(u^{\lambda})$ is a Cauchy sequence in
$L^2$. Let $\mu:=u^{\lambda}$, $v:=u^{\lambda'}$ and $w=\mu-v$, we
rely on the integral equation form
\begin{align}\label{xlem1}
\mu(t)=& \,U(t)u_0^{\lambda} -
\int_0^{t}U(t-\tau)F(\mu)(\tau)d\tau,
\end{align}
 and
\begin{align}\label{xlem2}
v(t)=& \,U(t)u_0^{\lambda'} - \int_0^{t}U(t-\tau)F(v)(\tau) d\tau,
\end{align}
where $F(u)$ is given by \eqref{NLF}. Thus \eqref{xlem1} and
\eqref{xlem2} imply
\begin{align*}
w(t)=& \,U(t)(u_0^{\lambda}-u_0^{\lambda'})-
\int_0^{t}U(t-\tau)(F(\mu)-F(v))(\tau) d\tau.
\end{align*}
Hence we have
\begin{align*}
\|w(t)\|_{L_x^2} \le &
\,\|u_0^{\lambda}-u_0^{\lambda'}\|_{L^2}+\left\|
\int_0^{t}U(t-\tau) (F(\mu)-F(v))(\tau) d\tau\right\|_{L_x^2}.
%
\end{align*}
We can suppose $a,b,c,d$ positive numbers, and using the
definition of $F(u)$
\begin{align}\label{xlem4}
\|\int_0^{t}&U(t-\tau)(F(\mu)-F(v))(\tau) d\tau \|_{L_x^2}
\le  c \int_0^{t}\left\||\mu|^2w+\mu v\overline{w}+|v|^2w\right\|_{L_x^2}d\tau  \nonumber\\
&  +\,d \int_0^{t}\left\|\mu v_x\overline{w}+w\overline{v}v_x\right\|_{L_x^2}d\tau+e \int_0^{t}\left\|w \mu\overline{\mu}_x+w v\overline{\mu}_x\right\|_{L_x^2}d\tau \nonumber\\
& +\|\int_0^{t}U(t-\tau)( d |\mu|^2w_x+ e v^2\overline{w}_x)(\tau) d\tau\|_{L_x^2}\nonumber\\
=: & \,I_{c}+I_{d}+I_{e}+I_{d e},
\end{align}
with obvious notation. Applying the H\"older inequality, the
Sobolev immersion $\|u(t)\|_{L_x^{\infty}} \le C \; \|
u(t)\|_{H^{1/2+}}$ and conservation law in $\dot{H}^1$, it follows
that
\begin{align*}
I_{c} \le & c\int_0^{t}(\|\mu\|_{L_x^{\infty}}+\|v\|_{L_x^{\infty}})^2\|w\|_{L_x^2}d\tau\le 2c\int_0^{t}(\|\mu\|_{H^1}^2+\|v\|_{H^1}^2)\|w\|_{L_x^2}d\tau\\
\le & C c (\|u_0^{\lambda}\|_{H^1}^2+\|u_0^{\lambda}\|_{L^2}^6+\|u_0^{\lambda'}\|_{H^1}^2+\|u_0^{\lambda'}\|_{L^2}^6)\int_0^{t}\|w\|_{L_x^2}d\tau\\
\le & 2C c
(\|u_0\|_{H^1}^2+\|u_0\|_{L^2}^6)\int_0^{t}\|w\|_{L_x^2}d\tau.
\end{align*}
Analogously, using H\"older inequality, the immersion
$\|u(t)\|_{L_x^{\infty}} \le C\| u(t)\|_{H^{1/2+}}$ and
conservation laws in $\dot{H}^1$ and $\dot{H}^2$, it follows that
\begin{align*}
I_{d} \le & d\int_0^{t}(\|\mu\|_{L_x^{\infty}}+\|v\|_{L_x^{\infty}})\|v_x\|_{L_x^{\infty}}\|w\|_{L_x^2}d\tau\\
 \le &d\int_0^{t}(\|\mu\|_{H^1}+\|v\|_{H^1})\|v\|_{H^2}\|w\|_{L_x^2}d\tau \\
\le & C d (\|u_0^{\lambda}\|_{H^1}+\|u_0^{\lambda}\|_{L^2}^3+\|u_0^{\lambda'}\|_{H^1}+\|u_0^{\lambda'}\|_{L^2}^3)\Omega_0\int_0^{t}\|w\|_{L_x^2}d\tau\\
\le & 2C \,d
(\|u_0\|_{H^1}+\|u_0\|_{L^2}^3)\Omega_0\int_0^{t}\|w\|_{L_x^2}d\tau,
\end{align*}
where $\Omega_0=\Omega_0(\|u_0\|_{L^2},\|u_0\|_{\dot{H}^1},
\|u_0\|_{\dot{H}^2}, T)$. Similarly we obtain
\begin{align*}
I_{e} \le & \,C \,e
(\|u_0\|_{H^1}+\|u_0\|_{L^2}^3)\Omega_0\int_0^{t}\|w\|_{L_x^2}d\tau.
\end{align*}
Now, we estimate $I_{d e}$. From \eqref{4tt}, we have
\begin{align*}
I_{d e} = &
\|\int_0^{t}U(t-\tau)\left(d(|\mu|^2w)_x-d(|\mu|^2)_xw+e
(v^2\overline{w})_x-2e v v_x \overline{w} \right)(\tau) d\tau\|_{L_x^2}\\
 \le &\| \partial_x \int_0^{t}U(t-\tau)(d|\mu|^2w+e v^2\overline{w})d\tau\|_{L_x^2}
 + \int_0^{t}\|d(|\mu|^2)_xw-2e v v_x \overline{w}\|_{L_x^2}d\tau \\
\le & C \,\|d|\mu|^2w+e v^2\overline{w} \|_{L_x^1 \mathcal{L}_t^2}
+ c \,(d+e) (\|u_0\|_{H^1}+\|u_0\|_{L^2}^3)\Omega_0\int_0^{t}\|w\|_{L_x^2}d\tau\\
\le & C \,|d|\|\mu\|_{L_x^4 \mathcal{L}_t^{\infty}}^2
\|w\|_{L_x^2 \mathcal{L}_t^{2}}+c|e| \|v\|_{L_x^4 \mathcal{L}_t^{\infty}}^2 \|w\|_{L_x^2 \mathcal{L}_t^{2}}\\
& + C \,(d+e) (\|u_0\|_{H^1}+\|u_0\|_{L^2}^3)\Omega_0\int_0^{t}\|w\|_{L_x^2}d\tau\\
\le &  C\,(d+e) \Big((\|\mu\|_{L_x^4 \mathcal{L}_t^{\infty}}^2
+\|v\|_{L_x^4\mathcal{L}_t^{\infty}}^2)\|w\|_{L_x^2\mathcal{L}_t^{2}}
\\
& +
(\|u_0\|_{H^1}+\|u_0\|_{L^2}^3)\Omega_0\int_0^{t}\|w\|_{L_x^2}d\tau
\Big).
\end{align*}
Applying Proposition \ref{eq:x33k.2}, we conclude
\begin{align*}
I_{d e} \le C \,(d+e) \left(\Omega_1\|w\|_{L_x^2
\mathcal{L}_t^{2}}+
(\|u_0\|_{H^1}+\|u_0\|_{L^2}^3)\Omega_0\int_0^{t}\|w\|_{L_x^2}d\tau\right),
\end{align*}
where $\Omega_1=\Omega_1(\|u_0\|_{L^2},\|u_0\|_{\dot{H}^1},
\|u_0\|_{\dot{H}^2}, T)$. Finally, we have
\begin{align}\label{xlem6}
\|w(t)\|_{L_x^2}\le \,\|u_0^{\lambda}-u_0^{\lambda'}\|_{L^2}+C
\Gamma_1 \int_0^t \|w(\tau)\|_{L^2}d\tau+C(d+e)\Omega_1
\|w\|_{L_x^2 \mathcal{L}_t^2},
\end{align}
where $\Gamma_1=2 C c \Gamma^2 + C(d+e)\Gamma \Omega_0$ and
$\Gamma=\|u_0\|_{H^1}+\|u_0\|_{L^2}^3$. Moreover, since
$$\int_0^t \|w(\tau)\|_{L^2}d\tau \le T^{1/2}\|w\|_{L_x^2
\mathcal{L}_t^2},$$
it follows from inequality \eqref{xlem6} that
\begin{align*}
\|w(t)\|_{L_x^2}^2\le
\,C\|u_0^{\lambda}-u_0^{\lambda'}\|_{L^2}^2+C \Gamma_1^2 T
\int_0^t \|w(\tau)\|_{L^2}^2d\tau+C\Omega_1^2 \int_0^t
\|w(\tau)\|_{L^2}^2d\tau,
\end{align*}
and using Gronwall's inequality
\begin{align*}
\|w(t)\|_{L_x^2}^2\le \,C\|u_0^{\lambda}-u_0^{\lambda'}\|_{L^2}^2
e^{Ct(\Gamma_1^2 T +\Omega_1^2)}.
\end{align*}
Consequently, $(u^{\lambda})$ is a Cauchy sequence in $L^2$, and
hence $u^{\lambda} \rightarrow u \in L^2$.

\medskip
2. Let $s \in (0,2)$, the interpolation in Sobolev spaces shows
that
\begin{align*}
\|u^{\lambda} - u^{\lambda'}\|_{H^{s}} \le & \|u^{\lambda} -
u^{\lambda'}\|_{L^2}^{1-s/2}
\|u^{\lambda} - u^{\lambda'}\|_{H^2}^{s/2}\\
 \le & \|u^{\lambda} - u^{\lambda'}\|_{L^2}^{1-s/2} \Omega_2,
\end{align*}
where $\Omega_2=\Omega_2(\|u_0\|_{L^2},\|u_0\|_{\dot{H}^1},
\|u_0\|_{\dot{H}^2}, T)$. Hence we have
\begin{equation}\label{cv1}
u^{\lambda} \rightarrow u  \quad\textrm{in} \,\,H^{s}\quad
\textrm{for all}
 \,\, s \in [0,2).
\end{equation}
Observe that the conservations laws in $L^2$ and $\dot{H}^1$ for
$u^{\lambda}$ implies that the limit $u$ also satisfies:
\begin{equation}\label{cv2}
\|u(t)\|_{L^2}=\|u(0)\|_{L^2}
\end{equation}
and
\begin{equation}\label{cv3}
\|u(t)\|_{\dot{H}^1} \le \|u(0)\|_{\dot{H}^1}+C\|u(0)\|_{L^2}^3.
\end{equation}
Moreover, the conserved quantity \eqref{CQ3} gives
\begin{align*}
\|u^{\lambda}_{xx}(t)\|_{L^2}^2 \le
   \big(\|u^{\lambda}_{xx}(0)\|^2_{L^2} + C_4\big) \, (1+T) \, e^{C_4 T},
\end{align*}
where the positive constants $C_4 = C_4(\|u^{\lambda}_0\|_{L^2},
\|u^{\lambda}_{0x}\|_{L^2})$. Thus
\begin{align}\label{clw1}
\|u^{\lambda}_{xx}(t)\|_{L^2}^2 \le C
   \big(\|u_{xx}(0)\|_{L^2},\|u_{x}(0)\|_{L^2},\|u(0)\|_{L^2},T).
\end{align}
Applying the Banach-Alaoglu Theorem, there exist a subsequence
already denoted by $u^{\lambda}$, and a function $\tilde{u} \in
H^2$, such that \beq \label{limf1} u^{\lambda}\rightharpoonup
\tilde{u}, \quad \textrm{in}\,\, H^2. \enq
Therefore, $u^{\lambda}\rightharpoonup \tilde{u}$ in $H^1$. On the
other hand, by \eqref{cv1} we have $u^{\lambda}\rightarrow u$ in
$H^1$ and thus $u^{\lambda}\rightharpoonup u$ in $H^1$.
Consequently, the uniqueness of the limit gives $u=\tilde{u} \in
H^2$. The inequality \eqref{clw1} and the limit \eqref{limf1}
implies
\begin{align}\label{clw2}
\|u_{xx}(t)\|_{L^2}= &\|\tilde{u}_{xx}(t)\|_{L^2}\le \liminf
\|u^{\lambda}_{xx}(t)\|_{L^2}^2 \nonumber \\ \le & C
   \big(\|u_{xx}(0)\|_{L^2},\|u_{x}(0)\|_{L^2},\|u(0)\|_{L^2},T).
\end{align}
%

\medskip
3. Now we will prove that $u$ is a solution of \eqref{IVP}. Using
Duhamel principle, we will show that, for each $t \in [-T,T]$,
$u(t)= \mathcal{L}(u)(t)$, where
\begin{align*}
\mathcal{L}(u)(t):= U(t)u(0,x) -
\int_0^{t}U(t-\tau)F(u)(\tau)d\tau.
\end{align*}
Let $\varpi:= \mathcal{L}(u^{\lambda})-\mathcal{L}(u)$, then
\begin{align}\label{clw0}
\|\varpi\|_{L^2} \le \|u^{\lambda}(0,x) - u(0,x)\|_{L^2}+\|
\int_0^t U(t-t')(F(u^{\lambda})-F(u))(t')dt'\|_{L^2}.
\end{align}
In the same way as in \eqref{xlem4}, we get
\begin{align}\label{1xlem4}
\|\int_0^{t}&U(t-\tau)(F(u^{\lambda})-F(u))(\tau) d\tau \|_{L_x^2}
\le  c\int_0^{t}\left\||u|^2w+uv\overline{w}+|v|^2w\right\|_{L_x^2}d\tau  \nonumber\\
&  +\,d
\int_0^{t}\left\|uv_x\overline{w}+w\overline{v}v_x+|u|^2w_x\right\|_{L_x^2}d\tau
+e \int_0^{t}\left\|w u\overline{u}_x+w v\overline{u}_x+v^2\overline{w}_x\right\|_{L_x^2}d\tau \nonumber\\
&=:  c\,I_{c}+d J_{d}+e J_{e},
\end{align}
where $v=u^{\lambda}$ and $w=u^{\lambda}-u$. We begin estimating
\begin{align}
I_{c} \le&  C(\|u\|_{L_t^{\infty}H^1}+\|v\|_{L_t^{\infty}H^1})^2\int_0^t \|w\|_{L_x^2}d\tau\nonumber\\
&\le
(\|u(0)\|_{H^1}+\|u_0^{\lambda}\|_{H^1}+\|u(0)\|_{L^2}^3+\|u_0^{\lambda}\|_{L^2}^3)^2
\int_0^t \|w\|_{L_x^2}d\tau.
\end{align}
Therefore, $I_{c} \to 0$ as $\lambda \to \infty$.  By
\eqref{CQ1}-\eqref{CQ3} and \eqref{cv1}-\eqref{cv3}
\begin{align}
J_{d} \le &  C(\|u\|_{L_t^{\infty}H^1}+\|v\|_{L_t^{\infty}H^1})\|v\|_{L_t^{\infty}H^2}\int_0^t \|w\|_{L_x^2}d\tau\nonumber\\
& +\|u\|_{L_t^{\infty}H^1}^2\int_0^t \|w_x\|_{L_x^2}d\tau.
\end{align}
Hence $J_{d} \to 0$ as $\lambda \to \infty$. Similarly by
\eqref{clw2}, we have
\begin{align}\label{clw4}
J_{e} \le&  C(\|u\|_{L_t^{\infty}H^1}+\|v\|_{L_t^{\infty}H^1})\|u\|_{L_t^{\infty}H^2}\int_0^t \|w\|_{L_x^2}d\tau\nonumber\\
& +\|v\|_{L_t^{\infty}H^1}^2\int_0^t \|w_x\|_{L_x^2}d\tau
\rightarrow 0.
\end{align}
Then, combining \eqref{cv1} and \eqref{clw0}-\eqref{clw4} and
passing to the limit as $\lambda \to \infty$
$$
\varpi
=\mathcal{L}(u^{\lambda})-\mathcal{L}(u)=u^{\lambda}-\mathcal{L}(u)
\rightarrow 0
\quad \textrm{in}\,\,L^2.
$$
The uniqueness of limit implies that $u=\mathcal{L}(u)$.
\end{proof}

\begin{remark}
{\bf 1} Observe that the proof of Lemma \ref{otg} gives another
way to prove the global well-posedness of \eqref{IVP} in $H^2$.
For instance, in order to show the persistence we proceed as
follow:

We claim that $u \in C([0,T], H^s)$, for $s \in [0,2]$. Indeed,
let $t_n \to t$ in $[0,T]$, and using the Duhamel's formula we
have
\begin{align*}
\|u(t_n)-u(t)\|_{H^2}\le& \|U(t_n)u_0-U(t)u_0\|_{H^2}\\
&+\|\int_0^{t_n}U(t_n-t')F(u)(t')dt'_0-\int_0^{t}U(t-t')F(u)(t')dt'\|_{H^2}\\
&=: L_1+L_2,
\end{align*}
with the obvious notation. In $L_1$, by the Dominated Convergence
Theorem, passing to the limit as $n \to \infty$,
\begin{align*}
L_1^2=\int_{\R}(1+\xi^2)2|e^{it_n \phi(\xi)}-e^{it \phi(\xi)}|^2
|\widehat{u}_0(\xi)|^2 d \xi \to 0,
\end{align*}
where $\phi(\xi)=a\xi^2+b\xi^3$. In $L_2$ also by Dominated
Convergence Theorem, we have
\begin{align*}
L_2 \le& \|\int_{t_n}^{t}U(t_n-t')F(u)(t')dt'_0\|_{H^2}+\|\int_0^{t}\left(U(t_n)\psi(t')dt'-U(t)\psi(t')\right)dt'\|_{H^2}\\
\le&
\|\int_{t_n}^{t}U(-t')F(u)(t')dt'_0\|_{H^2}+\|\int_0^{t}\left(U(t_n)\psi(t')dt'-U(t)\psi(t')\right)dt'\|_{H^2}\to
0,
\end{align*}
where $\psi(t')=U(-t')F(u)(t')$ and this proves that $u \in
C([0,T], H^2)$.

\bigskip
{\bf 2}\,\, In Lemma \ref{otg} we can consider $u_0^{\lambda}(x):=
\mathcal{F}^{-1}\left( \Chi_{\left\{|\xi|\le \lambda \right\}}
\widehat{u_0}(\xi)\right)(x)$ $\big($or $u_0^{\lambda}(x)=
\mathcal{F}^{-1}\left(
 \psi_{\lambda}(\xi)\widehat{u_0}(\xi)\right)(x)$ where $\psi_{\lambda}$ is a continuous function with
 support in $[-2\lambda, 2\lambda]$ and such that $\psi_{\lambda}=1$ in $[-\lambda,
 \lambda]$$\big)$. Therefore, if
$u_0 \in H^s$ then
\begin{align*}
\int |\widehat{u_0^{\lambda}}-\widehat{u_0}|^2(\xi)d\mu_{\theta}=
&\int |\Chi_{\left\{|\xi|> \lambda
\right\}}\widehat{u_0}|^2(\xi)d\mu_{\theta}
 =  \int_{|\xi|> \lambda}|\widehat{u_0}|^2(\xi)d\xi d\mu_{\theta}\rightarrow 0,
\end{align*}
when $\lambda \to \infty$. In this case, Paley-Wiener Theorem
implies that the initial data $u_0^{\lambda}$ in \eqref{reivaxIVP}
has an analytic continuation to an entire analytic $($in x$)$
function. On the other hand, by \cite{K-M} the solution
$u^{\lambda}$ of the IVP \eqref{reivaxIVP} also is an entire
analytic function. Hence we have the following :
\begin{corollary}If $u_0 \in H^2$ and $u(t)$ is the global solution of the IVP
\eqref{reivaxIVP} associated with the initial data $u_0$, then
there exists a sequence of entire analytic functions $u^{\lambda}$
such that $u^{\lambda}(t) \to u(t)$ in $H^{s}$, with $s \in
[0,2)$.
\end{corollary}%

\bigskip
{\bf 3}\,\,
If  $u_0 \in {L^2(d\dot{\mu}_{\theta})}$, $\theta \in [0,1]$,
$\lambda > 0$ and $u_0^{\lambda}(x)= \mathcal{F}^{-1}\left(
\Chi_{\left\{|\xi|< \lambda \right\}} \widehat{u_0}\right)(x)$,
then

\begin{equation}
\label{2.40}
    \|u_0^\lambda\|_{L^2(d\dot{\mu}_{\theta})} \leq
    \|u_0\|_{L^2(d\dot{\mu}_{\theta})}.
\end{equation}
In fact,
if $\theta=0$, \eqref{2.40} is a direct consequence of
Plancherel's theorem and definition of $u_0^\lambda$.
If $\theta=1$, using properties of Fourier transform we obtain
$$
|\widehat{x u_0^{\lambda}}(\xi)|=|\partial_{\xi}
\widehat{u_0^{\lambda}}(\xi)|=|\Chi_{\left\{|\xi|< \lambda
\right\}} \partial_{\xi} \widehat{u_0}(\xi)|=\Chi_{\left\{|\xi|<
\lambda \right\}}|\widehat{x u_0}(\xi)|.
$$
Thus by Plancherel's equality
\begin{equation*}
\int_{\R}x^2|u_0^{\lambda}(x)|^2 dx = \int_{\R}|\widehat{x
u_0^{\lambda}}(\xi)|^2 d\xi \le \int_{\R}|\widehat{x u_0}(\xi)|^2
d\xi=\int_{\R}|x u_0(x)|^2 dx.
\end{equation*}
When $\theta \in (0,1)$, we obtain \eqref{2.40} by interpolation
between the cases $\theta=0$ and $\theta=1$, see \cite{B-L}.
\end{remark}

\bigskip
\subsection{Main result}

Now, we state our main theorem of global existence:
\begin{theorem}\label{t1.2}
The IVP (\ref{IVP}) is globally well-posed in
$\mathcal{X}^{2,\theta}$ for any  $0\le \theta \le 1$ fixed.
Moreover, the solution $u$ of (\ref{IVP}) satisfies, for each $t
\in [-T,T]$
\begin{equation*}
  \|u(t)\|_{L^2(d\dot{\mu}_\theta)}^2  \leq C \; \Big(\|u_0\|_{L^2}^2
  + \|u_0\|_{L^2(d\dot{\mu}_\theta)}^2 \, + 1 \Big),
\end{equation*}
where $C=  C
(\theta,\|u(t)\|_{H^s},\|u(0)\|_{L^2},\|u_x(0)\|_{L^2},\|u_{xx}(0)\|_{L^2},T),
s> 1/2$.
\end{theorem}

\begin{proof}
Let $T>0$ and $u_0 \in \mathcal{X}^{2,\theta}$, $u_0 \neq 0$,
$\theta \in [0,1]$, we know that that there exists an function $u
\in C([-T,T], H^{2})$ such that the IVP \eqref{IVP} is global
well-posed in $H^{2}$. Is well know that $\mathbf{S}(\R)$ is dense
in $\mathcal{X}^{s,\theta}$. Then for $u_0 \in
\mathcal{X}^{2,\theta}$ there exist a sequence $(u_0^{\lambda})$
in $\mathbf{S}(\R)$ such that \be\label{converg} u_0^{\lambda} \to
u_0 \quad \textrm{in}\,\, \mathcal{X}^{2,\theta}. \ee
By \eqref{converg} and Lemma \ref{otg} the sequence of solutions
$u^{\lambda}(t)$ associated to IVP \eqref{reivaxIVP} and with
initial data $u_0^{\lambda}$ satisfy
\be\label{converg1}
 \sup_{t \in [-T,T]} \|u^{\lambda}(t)
 -u(t)\|_{H^s} \stackrel{\lambda \to \infty}{\rightarrow} 0
 \quad s \in [0,2).
 \ee
Suppose temporarily that the solutions $u^{\lambda}$ of the IVP
\eqref{reivaxIVP} satisfy the conditions (i)-(iv) of Section
\ref{CL}. Therefore Lemma \ref{1CL} gives
\begin{equation*}
\int_\RR |\xi|^{2 \theta} |u^{\lambda}(t,\xi)|^2 \, d\xi \leq C \;
\big(\int_\RR |u^{\lambda}(0,\xi)|^2 \,
  d\xi + \int_\RR |\xi|^{2 \theta} |u^{\lambda}(0,\xi)|^2 \,  d\xi + 1 \big),
\end{equation*}
where $C=  C
(\theta,\|u^{\lambda}(t)\|_{H^s},\|u^{\lambda}(0)\|_{L^2},\|u^{\lambda}_x(0)\|_{L^2},
\|u^{\lambda}_{xx}(0)\|_{L^2},T)$, $s \in (1/2,2)$, taking the
limit when $\lambda \to \infty$, \eqref{converg1} implies
\begin{equation*}
\int_\RR |\xi|^{2 \theta} |u(t,\xi)|^2 \, d\xi \leq C \;
\big(\int_\RR |u(0,\xi)|^2 \,
  d\xi + \int_\RR |\xi|^{2 \theta} |u(0,\xi)|^2 \,  d\xi + 1 \big),
\end{equation*}
where $C=  C
(\theta,\|u(t)\|_{H^s},\|u(0)\|_{L^2},\|u_x(0)\|_{L^2},\|u_{xx}(0)\|_{L^2},T)$.
Thus $u(t) \in \mathcal{X}^{2,\theta}$, $\theta \in [0,1]$, $t \in
[-T,T]$, which proves the persistence. The global well-posedness
theory in $H^{2}$ implies the uniqueness and continuous dependence
upon the initial data in $H^2$, therefore is sufficient prove
continuous dependence in the norm
$\|\cdot\|_{L^2(d\dot{\mu}_{\theta})}$. Let $u(t)$ and $v(t)$ be
two solutions in $\mathcal{X}^{2,\theta}$, $\theta \in [0,1]$ of
the IVP \eqref{IVP} with initial dates $u_0$ and $v_0$
respectively, let $u^{\lambda}(t)$, $v^{\lambda}(t)$ be the
solutions of the IVP \eqref{reivaxIVP} with initial dates
$u_0^{\lambda}$ and $v_0^{\lambda}$ respectively such that
$u_{0}^{\lambda}, v_{0}^{\lambda} \in \mathbf{S}(\R)$,
$u_{0}^{\lambda} \to u_0$, $v_{0}^{\lambda} \to v_0$ in
$\mathcal{X}^{2,\theta}$ and with $\lambda >> 1$, we have
\begin{align*}
\|u(t)-v(t)\|_{L^2(d\dot{\mu}_{\theta})}\le & \|u(t)-u^{\lambda}(t)\|_{L^2(d\dot{\mu}_{\theta})}+\|u^{\lambda}(t)-v^{\lambda}(t)\|_{L^2(d\dot{\mu}_{\theta})}\\
&+\|v^{\lambda}(t)-v(t)\|_{L^2(d\dot{\mu}_{\theta})}.
\end{align*}
 Convergence in \eqref{converg1}
 implies for $\lambda >>1$ that
 $$|u(x,t)-u^{\lambda}(x,t)|\le 2 |u(x,t)| \quad \textrm{and} \quad \quad |v(x,t)-v^{\lambda}(x,t)|\le 2 |v(x,t)|,$$
and the Dominated Convergence Lebesgue's Theorem gives
 \begin{align*}\|u(t)-u^{\lambda}(t)\|_{L^2(d\dot{\mu}_{\theta})}\to 0 \quad \textrm{and} \quad\|v^{\lambda}(t)-v(t)\|_{L^2(d\dot{\mu}_{\theta})}\to 0.
\end{align*}
Let $w^\lambda:=u^{\lambda}-v^{\lambda}$, then $w^\lambda$
satisfies the equation
\begin{align*}
w^\lambda_t &+ i
aw^\lambda_{xx}+bw^\lambda_{xxx}+c(|u^{\lambda}|^{2}w^\lambda+u^{\lambda}v^{\lambda}\bar{w}^{\lambda}+|v|^{2}w^\lambda)
\\
&+d(u^{\lambda}v^{\lambda}_x\bar{w}^\lambda+ w^\lambda
\bar{v}^{\lambda}v^{\lambda}_x+|u^{\lambda}|^{2}w^\lambda_{x})
+e(w^\lambda u^{\lambda}\bar{u}^{\lambda}_x+w^\lambda
v^{\lambda}\bar{u}^{\lambda}_x+(v^{\lambda})^{2}\bar{w}^\lambda_{x})=0.
\end{align*}
Then, we multiply the above equation by $\bar{w}^\lambda$,
integrate on $\R$ and take two times the real part, to obtain
$$
  \partial_t \int_\R |w^\lambda(t,x)|^2 \, dx \leq
  h(\|u_0\|_{H^2},\|v_0\|_{H^2}) \,
  \int_\R |w^\lambda(t,x)|^2 \, dx,
$$
where we have used convergence \eqref{converg1}, Lema \ref{LPE}
and $h$ is a polynomial function with $h(0,0)= 0$. Therefore, by
Gronwall's Lema, we have
$$
  \|w^\lambda(t)\|_{L^2} \leq \exp\big(T \,
  h(\|u_0\|_{H^2},\|v_0\|_{H^2})\big) \, \|w^\lambda_{0}\|_{L^2},
$$
which gives the continuous dependence in case $\theta= 0$.

Now, when $\theta=1$ a similar argument as used in the proof of
Proposition \ref{xav1} gives
$$
  \|w^\lambda(t)\|_{L^{2}(d\dot{\mu})} \leq \exp\big(T \,
  h_1(\|u_0\|_{H^2},\|v_0\|_{H^2})\big) \,
  \Big(\|w^\lambda_{0}\|_{L^{2}(d\dot{\mu})} + h_1(\|u_0\|_{H^2},\|v_0\|_{H^2}) \Big),
$$
where $h_1$ is a continuous function with $h_1(0,0)=0$.

Consequently, applying the Abstract Interpolation Lemma, we obtain
the continuous dependence for $\theta \in (0,1)$, where we have
assumed temporarily that the family $(w^\lambda)$ satisfies the
hypothesis of the Abstract Interpolation Lemma.

\bigskip
Finally we prove that the sequence of solutions
$(u^{\lambda_n}(t))$ satisfy the conditions (i)-(iv). Similarly,
we could obtain for the sequence $(w^{\lambda_n}(t))$.

\medskip
Condition (i): There exists $N_0$ such that $\forall \lambda \ge
N_0$ and for all $t \in [-T,T]$,
$$
  \clg{L}^1(\{x \in \R; u^{\lambda}(t,x) \neq 0\})> 0.
$$
In fact, by contradiction we suppose that there exist sequences
$\lambda_n \to \infty$ and $t_0 \in [-T,T]$
such that, $u^{\lambda_n}(t_0,x)=0$ almost everywhere.
By convergence \eqref{converg1} we conclude that $u(t_0)=0$, the
uniqueness of the solution implies $u=0$, in particular $u_0=0$,
which is a contradiction.

In the following, we consider $\clg{A}= (u^\lambda)_{\lambda>
N_0}$, see Section \ref{CL}.

Condition (ii): Inequality \eqref{AIL0} is a consequence of the
conservation law in $L^2$ and \eqref{AIL1} is a consequence of the
Proposition \ref{xav1}.

Condition (iii): We prove \eqref{AILC3} by contradiction. If there
exists a $\td{\theta} \in [0,1]$, such that for all $\Theta> 0$,
there exist $\lambda_n> N_0$, $t_0 \in [-T,T]$, $\gamma_0 \in
(0,1)$, such that
\begin{equation*}
    \int_{\{|u^{\lambda_n}(t_0)|^2 < \Theta\}} |u^{\lambda_n}(t_0)|^2 \; d\dot{\mu}_\td{\theta}
    > \gamma_0 \int_\RR |u^{\lambda_n}(t_0)|^2 \;
    d\dot{\mu}_\td{\theta}.
\end{equation*}
Then, taking the limit as $\Theta \to 0^+$ in the above
inequality, we arrive to contradiction.

Condition (iv): We prove \eqref{AILC4} also by contradiction. If
for all $R> 0$ and each $\gamma_2 \in (0,1)$, there exists
$\lambda_n> N_0$, such that
\begin{equation*}
    \int_{\{|\xi| \ge R\}} |u^{\lambda_n}(0)|^2 \; d\dot{\mu}
    > \gamma_2 \int_\RR |u^{\lambda_n}(0)|^2 \;
    d\dot{\mu}, \quad \gamma_2 \in(0,1),
\end{equation*}
similarly passing to the limit as $R \to +\infty$, carry to a
contradiction, which proves the condition (iv).
\end{proof}

\section*{Acknowledgements}

The authors were partially supported by FAPERJ through the grant
E-26/ 111.564/2008 entitled {\sl ``Analysis, Geometry and
Applications''}, and by the National Counsel of Technological and
Scientific Development (CNPq). The former by the grant
303849/2008-8 and the second author by the grant 311759/2006-8.

\newcommand{\auth}{\textsc}


\end{document}